\documentclass[a4paper]{amsart}

\usepackage{verbatim}
\usepackage[dvips]{color}
\usepackage{amssymb}
\usepackage[active]{srcltx}
\usepackage[colorlinks,pdfpagelabels,pdfstartview = FitH,bookmarksopen = true,bookmarksnumbered = true,linkcolor = blue,plainpages = false,hypertexnames = false,citecolor = red] {hyperref}

\usepackage{todonotes}
\usepackage{amsrefs}

\title[Equivalence of solutions]{Equivalence of solutions to fractional $p$-Laplace type equations}

\author[Korvenp\"a\"a]{Janne Korvenp\"a\"a}
\address{Janne Korvenp\"a\"a\\
Department of Mathematics and Systems Analysis,
Aalto University
\\ P.O. Box 11100
FI-00076 Aalto,
Finland}
\email{janne.korvenpaa@aalto.fi}

\author[Kuusi]{Tuomo Kuusi}
\address{Tuomo Kuusi\\
Department of Mathematics and Systems Analysis,
Aalto University
\\ P.O. Box 11100
FI-00076 Aalto,
Finland}
\email{tuomo.kuusi@aalto.fi}

\author[Lindgren]{Erik Lindgren}
\address{Erik Lindgren \\ Department of Mathematics, KTH, 100 44 Stockholm, Sweden}
\email{eriklin@kth.se}

\newtheorem{theorem}{Theorem}[section]

\newtheorem{proposition}{Proposition}[section]
\newtheorem{lemma}{Lemma}[section]

\theoremstyle{definition}
\newtheorem{definition}{Definition}
\newtheorem{remark}{Remark}[section]

\numberwithin{equation}{section}

\def\eqn#1$#2${\begin{equation}\label#1#2\end{equation}}
\def\charfn_#1{{\raise1.2pt\hbox{$\chi_{\kern-1pt\lower3pt\hbox{{$\scriptstyle#1$}}}$}}}

\newcommand{\eps}{\varepsilon}

\def\dist{\operatorname{dist}}

\def\supp{\operatorname{supp}}

\def\er{\mathbb R}
\newcommand{\ern}{\mathbb{R}^n}

\DeclareMathOperator*{\essliminf}{ess\,lim\,inf}

\def\mean#1{\mathchoice%
          {\mathop{\kern 0.2em\vrule width 0.6em height 0.69678ex depth -0.58065ex
                  \kern -0.8em \intop}\nolimits_{\kern -0.4em#1}}%
          {\mathop{\kern 0.1em\vrule width 0.5em height 0.69678ex depth -0.60387ex
                  \kern -0.6em \intop}\nolimits_{#1}}%
          {\mathop{\kern 0.1em\vrule width 0.5em height 0.69678ex
              depth -0.60387ex
                  \kern -0.6em \intop}\nolimits_{#1}}%
          {\mathop{\kern 0.1em\vrule width 0.5em height 0.69678ex depth -0.60387ex
                  \kern -0.6em \intop}\nolimits_{#1}}}

\def\vintslides_#1{\mathchoice%
          {\mathop{\kern 0.1em\vrule width 0.5em height 0.697ex depth -0.581ex
                  \kern -0.6em \intop}\nolimits_{\kern -0.4em#1}}%
          {\mathop{\kern 0.1em\vrule width 0.3em height 0.697ex depth -0.604ex
                  \kern -0.4em \intop}\nolimits_{#1}}%
          {\mathop{\kern 0.1em\vrule width 0.3em height 0.697ex depth -0.604ex
                  \kern -0.4em \intop}\nolimits_{#1}}%
          {\mathop{\kern 0.1em\vrule width 0.3em height 0.697ex depth -0.604ex
                  \kern -0.4em \intop}\nolimits_{#1}}}

\newcommand{\aveint}[2]{\mathchoice%
          {\mathop{\kern 0.2em\vrule width 0.6em height 0.69678ex depth -0.58065ex
                  \kern -0.8em \intop}\nolimits_{\kern -0.45em#1}^{#2}}%
          {\mathop{\kern 0.1em\vrule width 0.5em height 0.69678ex depth -0.60387ex
                  \kern -0.6em \intop}\nolimits_{#1}^{#2}}%
          {\mathop{\kern 0.1em\vrule width 0.5em height 0.69678ex depth -0.60387ex
                  \kern -0.6em \intop}\nolimits_{#1}^{#2}}%
          {\mathop{\kern 0.1em\vrule width 0.5em height 0.69678ex depth -0.60387ex
                  \kern -0.6em \intop}\nolimits_{#1}^{#2}}}

\newtoks\by
\newtoks\paper
\newtoks\book
\newtoks\jour
\newtoks\yr
\newtoks\pages
\newtoks\vol
\newtoks\publ

\def\name[#1, #2]{#1 #2}
\def\ota{{\hbox{\bf ???}}}
\def\cLear{\by=\ota\paper=\ota\book=\ota\jour=\ota\yr=\ota
\pages=\ota\vol=\ota\publ=\ota}
\def\endpaper{\the\by, \textit{\the\paper},
{\the\jour} \textbf{\the\vol} (\the\yr), \the\pages.\cLear}
\def\endbook{\the\by, \textit{\the\book},
\the\publ, \the\yr.\cLear}
\def\endpap{\the\by, \textit{\the\paper}, \the\jour.\cLear}
\def\endproc{\the\by, \textit{\the\paper}, \the\book, \the\publ,
\the\yr, \the\pages.\cLear}

\def\vs{\vspace{0.85mm}}

\allowdisplaybreaks

\numberwithin{equation}{section}

\newcommand{\e}{\varepsilon}

\newcommand{\R}{\mathbb{R}}

\newcommand{\dd}{\partial}

\newcommand{\lap}{\Delta}

\renewcommand{\div}{\operatorname{div}}
\renewcommand{\L}{\mathcal{L}}

\newcommand{\K}{\operatorname{Ker}(\Lambda)}

\begin{document}


\begin{abstract}

In this paper, we study different notions of solutions of nonlocal and nonlinear equations of fractional $p$-Laplace type
\[
{\rm P.V.} \int_{\ern}\frac{|u(x)-u(y)|^{p-2}(u(x)-u(y))}{|x-y|^{n+sp}}\,dy = 0.
\]
Solutions are defined via integration by parts with test functions, as viscosity solutions or via comparison. Our main result states that for bounded solutions, the three different notions coincide.

\end{abstract}
\maketitle

\section{Introduction}

We are concerned with different notions of solutions of nonlocal and nonlinear equations modeled by the fractional $p$-Laplace equation
\begin{equation}\label{eq:fracp}
{\rm P.V.} \int_{\ern}\frac{|u(x)-u(y)|^{p-2}(u(x)-u(y))}{|x-y|^{n+sp}}\,dy = 0,
\end{equation}
where $s \in (0,1)$ and $p>1$.
In recent years there has been a surge of interest around equations related to \eqref{eq:fracp}. The aim of this paper is to prove that, under reasonable assumptions, the different notions of solutions to~\eqref{eq:fracp} are equivalent. The three different notions we aim to treat are:

\begin{enumerate}
\item \emph{Weak} solutions of \eqref{eq:fracp} as in Definition~\ref{def:weaksuper}. These arise naturally as minimizers of the Gagliardo seminorm, i.e., minimizers of 
\[
\int_{\ern}\int_{\ern}\frac{|u(x)-u(y)|^{p}}{|x-y|^{n+sp}}\,dx dy,
\]
and the concept of being a solution is defined trough the first variation. The direct methods of the calculus of variations easily provide existence and uniqueness.

\smallskip

\item The potential-theoretic $(s,p)$-\emph{harmonic} functions as in Definition~\ref{def:superharmonic}. These are defined via comparison with weak solutions and they naturally arise for example in the Perron method.

\smallskip

\item \emph{Viscosity} solutions of \eqref{eq:fracp} as in Definition~\ref{def:viscosity super}. The notion of viscosity solutions is based on the pointwise evaluation of the principal value appearing in~\eqref{eq:fracp}.
\end{enumerate}

Instead of~\eqref{eq:fracp}, we consider more general type equations
\[
{\rm P.V.} \int_{\ern}|u(x)-u(y)|^{p-2}(u(x)-u(y))K(x,y)\,dy = 0,
\]
where the kernel $K$ has growth similar to $|x-y|^{-n-sp}$, see Section \ref{s.defs} for precise definitions.
The weak solutions in (1), as well as potential-theoretic $(s,p)$-harmonic functions in (2) are well-defined for very general, merely measurable kernels, since there is a natural weak formulation behind as soon as $K(\cdot,\cdot)$ is symmetric. However, in order to obtain the equivalence between different notions of solutions, we are forced to assume that $K$ is translation invariant. This is essentially a necessity, as explained already in \cite{Sil06}. We call the collection of suitable kernels as $\K$, where $\Lambda$ is measuring the ellipticity.

The nonlocal integro-differential equation~\eqref{eq:fracp} can be seen as a fractional nonlocal counterpart to the usual $p$-Laplace equation
\[
\lap_p u =\div(|\nabla u|^{p-2}\nabla u )=0.
\]
In fact, it is shown in~\cite{IN10} that the fractional viscosity solutions of~\eqref{eq:fracp} converge to the ones of $p$-Laplacian as $s\to 1$. In the case of the $p$-Laplace equation, the equivalence of solutions was first proved in \cite{JLM01} (see also \cite{Lin86}), and a shorter proof was recently given in \cite{JJ12}. It is notable that both in~\cite{JLM01} and in~\cite{JJ12} there is a need for a technical regularization procedure via infimal convolutions, which can be completely avoided in the nonlocal case. 

\smallskip

Our main result, using the recent results in \cite{DKKP15} and \cite{KKP15}, states that solutions defined via comparison and viscosity solutions are exactly the same for the class of kernels $\K$, see Section~\ref{s.defs}. 

\begin{theorem}\label{thm:equivalence0} \label{thm:viscsuperharm} \label{thm:superharmvisc}
Suppose that the kernel $K$ belongs to $\K$. Then, a function $u$ is $(s,p)$-superharmonic in $\Omega$ if and only if it is an $(s,p)$-viscosity supersolution in $\Omega$.
\end{theorem}

\smallskip 

In case that the supersolution is bounded or belongs to the right Sobolev space, we have the full equivalence result.
\begin{theorem}\label{thm:equivalence}
Suppose that the kernel $K$ belongs to $\K$. 
Assume that $u$ is locally bounded from above in $\Omega$ or $u \in W^{s,p}_{\rm loc}(\Omega)$. Then the following statements are equivalent:
\begin{enumerate}

\item $u$ is the lower semicontinuous representative of a weak supersolution in $\Omega$.

\item $u$ is $(s,p)$-superharmonic in $\Omega$.

\item $u$ is an $(s,p)$-viscosity supersolution in $\Omega$.

\end{enumerate}
\end{theorem}

In particular, by the theorem above we obtain that a continuous bounded energy solution is a viscosity solution. Moreover, if a weak solution is trapped between two functions that are regular enough, then the principal value in~\eqref{eq:fracp} is well-defined and zero, see Proposition~\ref{prop:principal value}. As a matter of fact, Proposition~\ref{prop:principal value} and the main theorems above assert that if a lower semicontinuous supersolution touches a smooth function from above, then the principal value exists at that point and is nonnegative. 

From the very definition of viscosity supersolutions we see that there are no integrability or differentiability assumptions on them. In view of Theorem~\ref{thm:equivalence0}, we may directly apply~\cite[Theorem 1]{DKKP15} to obtain the following useful result for viscosity supersolutions, which is, starting from the definition of them, not completely obvious. 
 \begin{theorem} \label{thm:viscosity super}
Suppose that the kernel $K$ belongs to $\K$. If $u$ is an $(s,p)$-viscosity supersolution in an open set $\Omega$, then it has the following properties: 
\begin{itemize}
\item[(i)] {\bf Pointwise behavior}.
\begin{equation*} 
u(x)= \liminf_{y \to x} u(y) = \essliminf_{y \to x}u(y) \quad \text{for every } x\in\Omega.
\end{equation*}
\item[(ii)] {\bf Summability}. For 
\begin{equation*}  \label{e.bar qt}
\bar t := \begin{cases}
  \frac{n(p-1)}{n-sp}, & 1<p< \frac ns,   \\[1ex]
   +\infty , & p \geq \frac{n}s,
\end{cases} \qquad \bar q := \min\left\{\frac{n(p-1)}{n-s},p\right\},
\end{equation*}
and  $h \in (0,s)$, $t \in (0,\bar t)$ and $q \in (0,\bar q)$, $u \in W_{\rm loc}^{h,q}(\Omega) \cap L_{\rm loc}^{t}(\Omega) \cap L_{sp}^{p-1}(\R^n)$. 
\end{itemize}
\end{theorem}

\subsection{Some known results}

To the best of our knowledge, equations of the type \eqref{eq:fracp} were first considered in \cite{IN10}, where viscosity solutions are studied. Existence, uniqueness, and the convergence to the $p$-Laplace equation as the parameter $s\to 1$, are proved. There is a slight technical difference of the equations studied therein, which is that the kernel only has support in a ball of radius $\rho(x)$. In \cite{CLM12}, weak solutions of the equation are studied, mostly for $p$ large, in connection with optimal H\"older extensions.

When it comes to regularity theory, there are many recent results. In \cite{DKP15} and \cite{DKP14}, local H\"older regularity of weak solutions is studied, in the flavor of  De Giorgi-Nash-Moser. In \cite{Lin14} the local H\"older regularity for viscosity solutions is studied. Both of these studies were preceded and inspired by the methods and results in \cite{Kas09} and \cite{Sil06}, respectively. Very recently, the sharp regularity up to the boundary was obtained in \cite{IMS15}. We also seize the opportunity to mention the paper \cite{KMS15b}, in which regularity properties for equations like \eqref{eq:fracp} with a measure are studied. There has also been some recent progress in terms of higher integrability results; in \cite{KMS14}, \cite{KMS15}, and \cite{Coz15} the case $p=2$ is treated. In \cite{Sch15} and \cite{BL15} the general case when $p\geq 2$ is considered.

Worth mentioning is also that there are other ways of defining a nonlocal or a fractional version of the $p$-Laplace equation by developing the terms and then replacing this with a suitable nonlocal operator. This is the direction taken in \cite{BCF12a}, \cite{BCF12b}, and \cite{CJ15}. In \cite{BCF12a} and \cite{BCF12b} an interesting connection to a nonlocal tug-of-war is found and in \cite{CJ15} a connection to L\'evy processes is made. Note that the operators considered in these papers differ substantially from the operators considered in the present paper.

As mentioned earlier, the literature on equations like \eqref{eq:fracp} has literally exploded in recent years, hence we do not, in any way, claim to give a full account here.

\smallskip
The paper is organized as follows. In Section \ref{s.defs} below, we introduce the definitions and notations stating also some recent results. In Section \ref{Sec1}, we prove several properties for fractional $p$-Laplace type operators from the viscosity solution point of view. 
In Section \ref{s.comp}, we state and prove a weak comparison principle for viscosity solutions, which is one of the keys to our main result.
Finally, Section \ref{s.main} is devoted for proving the main result. 

\section{Definitions and notation} \label{s.defs}

In this section, we introduce and define the different notions of solutions and also some notation.

\subsection*{General notation}
Throughout the paper, $\Omega$ will denote a bounded and open set in $\ern$. If an open set $D$ is compactly contained in $\Omega$, we will write $D \Subset \Omega$. We will denote by $B_r(x)$, the ball of radius $r$ centered at the point $x$. When the center is the origin, we will suppress it from the notation: $B_r \equiv B_r(0)$. The positive and the negative part of a function $u$ will be used several times, defined as 
$$
u_+=\max\{u,0\},\quad u_-=\max\{-u,0\}.
$$

\smallskip

\subsection*{Kernels} Let us then give the description of suitable kernels. We say that the kernel $K\colon \ern \times \ern \to (0,\infty]$ belongs to $\K$, if it has the following four properties:
\begin{itemize}
\item[(i)] Symmetry: $K(x,y)=K(y,x)$ for all $x,y \in \ern$.

\item[(ii)] Translation invariance: $K(x+z,y+z)=K(x,y)$ for all $x,y,z \in \ern$, $x \neq y$.

\item[(iii)] Growth condition:
$
\Lambda^{-1} \leq K(x,y)|x-y|^{n+sp} \leq \Lambda $ for all $x,y \in \ern$, $x \neq y$. 

\item[(iv)] Continuity: The map $x \mapsto K(x,y)$ is continuous in  $\R^n \setminus \{y\}$. 
\end{itemize}
Above $\Lambda \geq 1$ is a constant. 
Notice that by symmetry also $y \mapsto K(x,y)$ is continuous in $\R^n \setminus \{x\}$ if $K \in \K$. In the case of the fractional $p$-Laplace equation, $K(x,y)=|x-y|^{-n-sp}$.  

\smallskip

 For $K(x,0)$ we often write $K(x)$, in short.
The operator $\L$ will denote the nonlocal and nonlinear operator associated with the kernel $K$ as follows
\begin{align}\label{eq:L fl}
\L u(x) := {\rm P.V.} \int_{\ern}|u(x)-u(y)|^{p-2}(u(x)-u(y))K(x,y)\,dy.
\end{align}
Here the symbol P.V. means ``in the principal value sense" and is often omitted when it is clear from the context.

\smallskip

\subsection*{Fractional Sobolev spaces}
A central role will naturally be played by the so-called
fractional Sobolev spaces (also known as Aronszajn, Gagliardo, or Slobodeckij spaces)
$W^{s,p}(\R^n)$ with $0<s<1$ and $1<p<\infty$. The norm is defined through
$$
\|u\|_{W^{s,p}(\R^n)}=  \|u\|_{L^{p}(\R^n)} + [u]_{W^{s,p}(\R^n)}, 
$$
where the quantity
$$
[u]_{W^{s,p}(\R^n)}=\left(\int_{\R^n}\int_{\R^n} \frac{|u(x)-u(y)|^{p}}{|x-y|^{n+sp}}\,dx dy\right)^\frac1p
$$
is called the Gagliardo seminorm of $u$.
The space $W^{s,p}(\Omega)$ is defined similarly and, as usual, $W^{s,p}_0(\Omega)$ is defined as the closure of $C_0^\infty(\Omega)$ with respect to the norm $\|\cdot \|_{W^{s,p}(\Omega)}$. We refer to the ``Hitchhiker's Guide to the Fractional Sobolev Spaces'', \cite{NPV12}, for most of the properties of fractional Sobolev spaces used in this paper.

\smallskip

\subsection*{Tail spaces}
Terms that we will refer to as tails appear often in nonlocal settings, and therefore we define
\begin{equation} \label{def_tail} 
{\rm Tail}(f;z,r) := \bigg(r^{sp} \int_{\ern \setminus B_r(z)} |f(x)|^{p-1} |x-z|^{-n-sp}\,dx \bigg)^{\frac{1}{p-1}}.
\end{equation}
The ``tail space" is defined accordingly
\begin{equation} \label{def_tailspace'}
L^{p-1}_{sp}(\R^n) = \Big\{ f \in L_{\rm loc}^{p-1}(\ern) : \int_{\ern} |f(x)|^{p-1} (1+|x|)^{-n-sp} \,dx < \infty \Big\},
\end{equation}
and it is easy to see that if $fÊ\in L^{p-1}_{sp}(\R^n)$, then ${\rm Tail}(f;z,r)$ is finite for all $z \in \R^n$ and $r \in (0,\infty)$. 
\smallskip

\subsection*{Notions of solutions}

We next introduce the different notions of solutions of the equation
\begin{equation}
\label{eq:maineq}
\L u = 0 \quad \text{in $\Omega$}.
\end{equation}

\begin{definition} \label{def:weaksuper}
A function $u\in W_{\textup{loc}}^{s,p}(\Omega) \cap L^{p-1}_{sp}(\ern)$ is a \emph{weak supersolution} of \eqref{eq:maineq} if
\[
\int_{\ern}\int_{\ern}|u(x)-u(y)|^{p-2}(u(x)-u(y))(\phi(x)-\phi(y))K(x,y)\,dx dy\geq 0
\]
for all nonnegative $\phi\in C_0^\infty(\Omega)$.
\end{definition}
A function $u$ is a \emph{weak subsolution} of \eqref{eq:maineq} if $-u$ is a weak supersolution. Moreover, $u$ is a \emph{weak solution} of \eqref{eq:maineq} if it is both a weak supersolution and a subsolution, or equivalently
\[
\int_{\ern}\int_{\ern}|u(x)-u(y)|^{p-2}(u(x)-u(y))(\phi(x)-\phi(y))K(x,y)\,dx dy = 0
\]
for all $\phi\in C_0^\infty(\Omega)$.

\begin{definition} \label{def:superharmonic}
We say that a function $u\colon\ern \to [-\infty,\infty]$ is an $(s,p)$-\emph{superharmonic} function in $\Omega$ if it satisfies the following four assumptions.
\begin{itemize}
\item[(i)] $u<+\infty$ almost everywhere in $\R^n$, and $u> -\infty$ everywhere in~$\Omega$. 
\item[(ii)] $u$ is lower semicontinuous in~$\Omega$.
\item[(iii)] $u$ satisfies the comparison in $\Omega$ against solutions, that is, if $D \Subset \Omega$ is an open set and $v \in C(\overline{D})$ is a weak solution in $D$ such that $u \geq v$ on $\partial D$ and almost everywhere in~$\R^n \setminus D$, then $u \geq v$ in $D$.
\item[(iv)] $u_-$ belongs to $L_{sp}^{p-1}(\ern)$.
\end{itemize}
\end{definition}
A function $u$ is $(s,p)$-\emph{subharmonic} in $\Omega$ if $-u$ is $(s,p)$-\emph{superharmonic} in $\Omega$. Moreover, $u$ is $(s,p)$-\emph{harmonic} in $\Omega$ if it is both $(s,p)$-subharmonic and $(s,p)$-superharmonic.

\begin{remark}
In the definition of $(s,p)$-superharmonic functions in \cite{DKKP15}, the comparison condition (iii) is against solutions that are bounded from above in $\ern$. However, our class of $(s,p)$-superharmonic functions is exactly the same according to \cite[Theorem 1(iii)]{DKKP15}.
\end{remark}

In order to define the notion of viscosity solutions for exponents in the range $p\leq \frac{2}{2-s}$, we need a more restricted class of test functions. Indeed, in this range the operator is singular, in the sense that it is not well defined even on smooth functions. For example, defining
\begin{equation*} 
u(x)=
\begin{cases} 
|x|^2, & x \in B_1, \\
1,  & x \in \R^n \setminus B_1,
\end{cases}
\end{equation*}
which is smooth close to origin, we have that the principal value $\mathcal L u(0)$ is finite if and only if~$p>\frac{2}{2-s}$.

When $x_0$ is an isolated critical point, in essence we would like to test viscosity solutions by merely using functions of the type $|x-x_0|^\beta$. However, we need some flexibility in the choice of test functions and this motivates the definition of $C^{2}_{\beta}$ below. One should anyway keep in mind that the space $C^{2}_{\beta}$ contains monomials like $|x-x_0|^\beta$ plus suitable perturbations. 

Let us introduce some notation. The set of critical points of a differentiable function $u$ and the distance from the critical points are denoted by
\[
N_u:=\{x \in \Omega : \nabla u(x)=0\}, \quad d_u(x):=\dist(x,N_u),
\]
respectively. Let $D \subset \Omega$ be an open set. We denote the class of $C^{2}$-functions whose gradient and Hessian are controlled by $d_u$ as
\begin{equation} \label{eq:C_beta^2}
C^{2}_{\beta}(D) := \left\{u \in C^{2}(D) : \sup_{x\in D} \left( \frac{\min\{d_u(x),1\}^{\beta-1}}{|\nabla u(x)|} +  \frac{|D^{2}u(x)|}{d_u(x)^{\beta-2}} \right) < \infty \right\}.
\end{equation}
The supremum in the definition is denoted by $\| \cdot \|_{C^{2}_{\beta}(D)}$.
Notice that, in particular, when $\beta \geq 2$, the function $\phi(x)=|x|^{\beta}$ is in the class $C^{2}_{\beta}$.

\begin{definition} \label{def:viscosity super}
We say that a function $u\colon \ern \to [-\infty,\infty]$ is an $(s,p)$-\emph{viscosity supersolution} in $\Omega$ if it satisfies the following four assumptions.
\begin{itemize}
\item[(i)] $u<+\infty$ almost everywhere in $\R^n$, and $u> -\infty$ everywhere in~$\Omega$.
\item[(ii)] $u$ is lower semicontinuous in $\Omega$.
\item[(iii)]
If $\phi \in C^{2}(B_r(x_0))$ for some $B_r(x_0) \subset \Omega$ such that $\phi(x_0)=u(x_0)$ and $\phi \leq u$ in $B_r(x_0)$, and one of the following holds
\begin{itemize}
\item[(a)]
$p>\frac{2}{2-s}$ or $\nabla \phi(x_0) \neq 0$,
\item[(b)]
$1<p \leq \frac{2}{2-s}$, $\nabla \phi(x_0) = 0$ such that $x_0$ is an isolated critical point of $\phi$, and $\phi \in C^{2}_{\beta}(B_r(x_0))$ for some $\beta > \tfrac{sp}{p-1}$,
\end{itemize}
then $\L \phi_r(x_0) \geq 0$, where
\[
\phi_r(x)=\begin{cases}
\phi(x), &  x \in B_r(x_0),  \\
u(x), & x \in \ern \setminus B_r(x_0).
\end{cases}
\]
\item[(iv)] $u_-$ belongs to $L_{sp}^{p-1}(\ern)$.
\end{itemize}
\end{definition}
A function $u$ is an $(s,p)$-\emph{viscosity subsolution} in $\Omega$ if $-u$ is an $(s,p)$-viscosity supersolution. Moreover, $u$ is an $(s,p)$-\emph{viscosity solution} in $\Omega$ if it is both an $(s,p)$-viscosity supersolution and a subsolution.

\smallskip

Let us say a few more words about the case $1<p \leq \frac{2}{2-s}$. Observe that if $x_0$ is a critical point of $\phi$, but not isolated, there is no test for $u$ at such points. 

\begin{remark} One might be tempted to allow only test functions with non-vanishing gradient at the testing point. However, this will, in constrast to the local case, lead to false solutions. Indeed, with this class of test functions, any function that is constant in $\Omega$ will be a solution to the equation, no matter what boundary values it takes outside of~$\Omega$. 
\end{remark}

\setcounter{equation}{0} \setcounter{theorem}{0}

\subsection*{Some recent results on nonlocal supersolutions and superharmonic functions}

We recall some results from \cite{DKKP15}.
First of them is a natural comparison principle between $(s,p)$-superharmonic and $(s,p)$-subharmonic functions.

\begin{theorem}[{\bf Comparison principle of $(s,p)$-harmonic functions}]{\rm(\cite[Theorem 16]{DKKP15}).} \label{comp principle}
Let $u$ be $(s,p)$-superharmonic in $\Omega$ and let $v$ be $(s,p)$-subharmonic in $\Omega$.
If $u \geq v$ almost everywhere in $\R^n \setminus \Omega$ and
\[
\liminf_{\Omega \ni y \to x} u(y) \geq \limsup_{\Omega \ni y \to x} v(y) \qquad \text{for all } x \in \partial\Omega
\]
such that both sides are not simultaneously $+\infty$ or $-\infty$, then $u \geq v$ in $\Omega$.
\end{theorem}


Weak supersolutions and $(s,p)$-superharmonic functions are closely related, as demonstrated in~\cite{DKKP15}. 
Bounded $(s,p)$-superharmonic functions are weak supersolutions and, on the other hand,
weak supersolutions have lower semicontinuous representatives that are $(s,p)$-superharmonic.
\begin{theorem}{\rm(\cite[Theorem 1(iv)]{DKKP15}).} \label{thm_supersuper}
Let $u$ be an $(s,p)$-superharmonic function in $\Omega$. If $u$ is locally bounded in $\Omega$ or belongs to $W^{s,p}_{\rm loc}(\Omega)$, then $u$ is a weak supersolution of \eqref{eq:maineq}.
\end{theorem}

\begin{theorem}{\rm(\cite[Theorem 12]{DKKP15}).} \label{thm_supersuper2}
Let $u$ be the lower semicontinuous representative of a weak supersolution of \eqref{eq:maineq} satisfying
\begin{equation} \label{eq:essliminf000}
u(x)=\essliminf_{y \to x}u(y) \qquad \text{for every } x\in\Omega.
\end{equation}
Then $u$ is an $(s,p)$-superharmonic function in $\Omega$.
\end{theorem}


\section{Auxiliary tools}\label{Sec1}
In this section, we gather some technical results needed in the sequel. In Subsection~\ref{ss.AI} we collect some elementary algebraic facts. In Subsection~\ref{ss.PV} we prove that the principal values are well-defined for functions that are smooth enough, and in Subsection~\ref{ss.CP}, in turn, we show that the operator $\L$ is locally uniformly continuous and stable with respect to smooth perturbations (cf. Lemma \ref{lemma:Lphi cont} and Lemma \ref{lemma:Lphi cont in eps}). 

\subsection{Algebraic inequalities} \label{ss.AI}

We start with a trivial observation.

\begin{lemma} \label{lemma:affines in kernel}
Let $\ell$ be an affine function and let $r \in(0,\infty)$. Then
\[
\int_{B_r(x) \setminus B_\eps(x)} |\ell(x) - \ell(y)|^{p-2}(\ell(x) - \ell(y)) K(x,y) \, dy = 0
\]
for all $\eps \in (0,r)$. 
\end{lemma}
We omit the obvious proof.
The following elementary results turn out to be useful in characterizing the principal values. 

\begin{lemma} \label{lemma:stupid integral} 
Let $p>1$ and $a,b\in \er $. 
Then
\[
 \int_0^1 |a +  b t|^{p-2} \, dt \leq   c_p (|a| + |b|)^{p-2}, \qquad c_p := \begin{cases}
 1, &  p\geq 2,  \\
 \dfrac{4^{2-p}}{p-1}, & 1<p<2.
\end{cases}
\]
\end{lemma}

\begin{proof}
Assume first $a \neq 0$. If $p \geq 2$, we easily obtain
\[
\int_0^1 |a +  b t|^{p-2} \, dt  \leq (|a| + |b|)^{p-2}.
\]
Let then $1<p<2$. There are two cases to consider.
If $|a| \geq 2 |b|$, then
\[
\int_0^1 |a + b t|^{p-2} \, dt \leq (|a|-|b|)^{p-2} \leq 2^{2-p} |a|^{p-2} \leq 4^{2-p}(|a|+|b|)^{p-2}.
\]
If, in turn, $|a| < 2|b|$, denote $\tilde a = a/b \in (-2,2)$ and compute
\begin{align*} \label{eq:} 
 \int_0^1 |a +  b t|^{p-2} \, dt  &= |b|^{p-2} \int_{\tilde a}^{\tilde a +1} |\tau|^{p-2} \, d\tau \\
 &= \frac{|b|^{p-2}}{p-1}
\begin{cases}
(|\tilde a|+1)^{p-1} - |\tilde a|^{p-1}  , & \tilde a \geq 0 ,    \\
 (\tilde a+1)^{p-1} +  |\tilde a|^{p-1}, &  -1 <\tilde a < 0 ,  \\
  |\tilde a|^{p-1} - (|\tilde a|-1)^{p-1}  , & \tilde a \leq - 1.
\end{cases} 
\end{align*}
It follows that
\[
\int_0^1 |a +  b t|^{p-2} \, dt \leq \frac{4^{2-p}}{p-1} (|a|+|b|)^{p-2}.
\]
Finally, if $a=0$, then clearly
\[
\int_0^1 |a +  b t|^{p-2} \, dt = \frac{|b|^{p-2}}{p-1} \leq c_p (|a| + |b|)^{p-2}.
\]
This finishes the proof.
\end{proof}

\begin{lemma} \label{lemma:stupid integral 2} 
Let $p>1$ and $a,b \in \er$. 
Then
\[
 \int_0^1 |a +  b t|^{p-2} \, dt \geq   c_p (|a| + |b|)^{p-2}, \qquad c_p := \begin{cases}
 1, &  1<p< 2,  \\
 \dfrac{4^{2-p}}{p-1}, & p \geq 2.
\end{cases}
\]
\end{lemma}

\begin{proof}
Assume first $a \neq 0$. If $1<p<2$, we easily obtain
\[
\int_0^1 |a +  b t|^{p-2} \, dt  \geq (|a| + |b|)^{p-2}.
\]
Let then $p \geq 2$. There are two cases to consider. 
If $|a| \geq 2 |b|$, then
\[
\int_0^1 |a + b t|^{p-2} \, dt \geq (|a|-|b|)^{p-2} \geq 2^{2-p} |a|^{p-2} \geq 4^{2-p}(|a|+|b|)^{p-2}.
\]
If, in turn, $|a| < 2|b|$, let $\tilde a = a/b \in (-2,2)$ and compute
\begin{align*}
 \int_0^1 |a +  b t|^{p-2} \, dt  &= |b|^{p-2}\int_{\tilde a}^{\tilde a + 1} |\tau|^{p-2} \, d\tau \\
&= \frac{|b|^{p-2}}{p-1}
\begin{cases}
(|\tilde a|+1)^{p-1} - |\tilde a|^{p-1}  , & \tilde a \geq 0 ,    \\
 (\tilde a+1)^{p-1} +  |\tilde a|^{p-1}, &  -1 <\tilde a < 0 ,  \\
  |\tilde a|^{p-1} - (|\tilde a|-1)^{p-1}  , & \tilde a \leq - 1.
\end{cases} 
\end{align*}
Switching back to $a$, we obtain
\[
\int_0^1 |a +  b t|^{p-2} \, dt \geq \frac{4^{2-p}}{p-1} (|a|+|b|)^{p-2}.
\]
Finally, if $a=0$, then
\[
\int_0^1 |a +  b t|^{p-2} \, dt = \frac{|b|^{p-2}}{p-1} \geq c_p (|a| + |b|)^{p-2}.
\]
This finishes the proof.
\end{proof}

\begin{lemma} \label{lemma:simple inequality}
Let $p>1$ and $a,b \in \er$. Then
\[
\left||a|^{p-2}a-|b|^{p-2}b\right| \leq c\,(|b|+|a-b|)^{p-2}|a-b|,
\]
where $c$ depends only on $p$.
\end{lemma}
\begin{proof}
Since
\[
\frac{d}{dt}\left(\left|ta+(1-t)b\right|^{p-2}\left(ta+(1-t)b\right)\right)=(p-1)(a-b)\left|ta+(1-t)b\right|^{p-2},
\]
we can estimate using Lemma \ref{lemma:stupid integral}
\begin{align*}
\left||a|^{p-2}a-|b|^{p-2}b\right| &= \left|(p-1)(a-b)\int_0^{1}|ta+(1-t)b|^{p-2}\,dt\right| \\
&\leq c_p(p-1)|a-b|(|b|+|a-b|)^{p-2},
\end{align*}
and the claim follows.
\end{proof}

The following estimate can be easily obtained using spherical coordinates.
\begin{lemma} \label{lemma:spherical estimate}
Let $e$ be a unit vector in $\ern$, let $p>1$, and let $a\geq 0$. Then
\[
\int_{S^{n}}(|e \cdot \omega| + a)^{p-2}\,d\omega \leq c\,(1+a)^{p-2},
\]
where $S^{n}$ is the unit sphere around the origin and $c$ depends only on $n$ and $p$.
\end{lemma}

\subsection{Principal values} \label{ss.PV}
The aim is now to prove that the principal value defining the operator $\L$ is well-defined when the involved functions are smooth enough. In order to accomplish this, we need uniform estimates on small balls. These are the following two lemmas. Throughout the section, we assume that $K \in \K$. 

\begin{lemma} \label{lemma:smallness of intBeps1}
Let $B_\eps(x) \subset D \Subset \Omega$ and let $u \in C^{2}(D)$. If we have $p>\frac{2}{2-s}$ or $D \Subset \{d_u>0\}$, then
\begin{align} \label{eq:intBeps bound1}
\left|{\rm P.V.}\int_{B_\eps(x)}|u(x)-u(y)|^{p-2}(u(x)-u(y))K(x,y)\,dy\right| \leq c_\eps,
\end{align}
where $c_\eps$ is independent of $x$  and $c_\eps \to 0$ as $\eps \to 0$.
\end{lemma}
\begin{proof} If $|\nabla u(x)|=0$ and $p>\frac{2}{2-s}$, the result is quite obvious using the fact that  $u \in C^2$ and $\nabla u(x) = 0$ implies
$$
|u(x)-u(y)|\leq C|x-y|^2,
$$ for some constant $C$. For this reason we only treat the case $|\nabla u(x)|\neq 0$. Throughout the proof, the constant $c$ will denote a constant depending on $n$, $p$, $s$, $\Lambda$, $\|u\|_{C^2(D)}$, and $D$.

Let $\ell(y):=u(x)+\nabla u(x)\cdot(y-x)$ be the affine part of $u$ near $x$. Denoting by $g(t):=|t|^{p-2}t$, we have by Lemma \ref{lemma:affines in kernel} and
Lemma \ref{lemma:simple inequality} that
\begin{align*}
&\left|\int_{B_\eps(x)}|u(x)-u(y)|^{p-2}(u(x)-u(y))K(x,y)\,dy\right| \\
&\qquad \leq \int_{B_\eps(x)}\left|g(u(x)-u(y))-g(\ell(x)-\ell(y))\right|K(x,y)\,dy \\
&\qquad \leq c\int_{B_\eps(x)}\big(|\ell(x)-\ell(y)|+|u(y)-\ell(y)|\big)^{p-2}|u(y)-\ell(y)|K(x,y)\,dy \\
&\qquad = c\int_{B_\eps(x)}\big(|\nabla u(x) \cdot(y-x)|+|u(y)-\ell(y)|\big)^{p-2}|u(y)-\ell(y)|K(x,y)\,dy.
\end{align*}
After introducing $z=y-x$ and switching to spherical coordinates, we obtain, with $\tau:=\sup_{D}|D^{2}u|$, 
\begin{align*}
&\left|\int_{B_\eps(x)}|u(x)-u(y)|^{p-2}(u(x)-u(y))K(x,y)\,dy\right| \\
&\qquad \leq c\int_{B_\eps(0)}\Big(|\nabla u(x) \cdot z|+ \sup_{|\xi-x|<|z|}|D^{2}u(\xi)||z|^{2}\Big)^{p-2}\!\!\!\! \sup_{|\xi-x|<|z|}|D^{2}u(\xi)||z|^{2}K(z)\,dz \\
&\qquad \leq c\int_0^{\eps}\int_{S^{n}}\bigg(|\nabla u(x) \cdot \omega|r+\tau r^{2}\bigg)^{p-2}\tau r^{2-n-sp+n-1}\,d\omega\,dr \\
&\qquad = c\,\tau \int_0^{\eps}\int_{S^{n}}\left(\frac{|\nabla u(x) \cdot \omega|}{|\nabla u(x)|}+\frac{\tau r}{|\nabla u(x)|}\right)^{p-2} |\nabla u(x)|^{p-2} d\omega\, r^{p(1-s)}\frac{dr}{r},
\end{align*}
where we used the monotonicity of $(a+b)^{p-2}b$ with respect to $b$ when $a,b \geq 0$ and the upper bound for $K(z)$. Applying Lemma \ref{lemma:spherical estimate}, we obtain
\begin{align} \label{eq:int0eps bound}
&\left|\int_{B_\eps(x)}|u(x)-u(y)|^{p-2}(u(x)-u(y))K(x,y)\,dy\right| \nonumber \\
&\qquad\leq c\,\tau \int_0^{\eps}\left(1+\frac{\tau r}{|\nabla u(x)|}\right)^{p-2} |\nabla u(x)|^{p-2} r^{p(1-s)}\frac{dr}{r}.
\end{align}
If $p\geq 2$, we obtain from \eqref{eq:int0eps bound} that
\begin{align*}
&\left|\int_{B_\eps(x)}|u(x)-u(y)|^{p-2}(u(x)-u(y))K(x,y)\,dy\right| \\
&\qquad\leq c\,\tau \int_0^{\eps}\left(1+\frac{\tau^{p-2}r^{p-2}}{|\nabla u(x)|^{p-2}}\right) |\nabla u(x)|^{p-2}  r^{p(1-s)}\frac{dr}{r} \\[1ex]
&\qquad\leq c\,\tau\sup_{D}|\nabla u|^{p-2}\eps^{p(1-s)}+c\,\tau^{p-1}\eps^{p-2+p(1-s)}.
\end{align*}
If $\frac{2}{2-s}<p<2$, \eqref{eq:int0eps bound} can simply be estimated as
\begin{align*}
&\left|\int_{B_\eps(x)}|u(x)-u(y)|^{p-2}(u(x)-u(y))K(x,y)\,dy\right| \\
&\qquad\leq c\,\tau\int_0^{\eps}\left(\frac{\tau r}{|\nabla u(x)|}\right)^{p-2}|\nabla u(x)|^{p-2}r^{p(1-s)}\frac{dr}{r} \\[1ex]
&\qquad\leq c\,\tau^{p-1}\eps^{p-2+p(1-s)}.
\end{align*}
Finally, if $1<p \leq \frac{2}{2-s}$ and $D \Subset \{x \in \Omega : d_u(x)>0\}$, then $\inf_{D}|\nabla u|>0$ and \eqref{eq:int0eps bound} leads to
\begin{align*}
&\left|\int_{B_\eps(x)}|u(x)-u(y)|^{p-2}(u(x)-u(y))K(x,y)\,dy\right| \\
&\qquad\leq c\,\tau\int_0^{\eps}|\nabla u(x)|^{p-2}r^{p(1-s)}\frac{dr}{r} \\
&\qquad\leq c\,\tau\sup_{D}|\nabla u|^{p-2}\eps^{p(1-s)}.
\end{align*}
In all cases it is now straightforward to check the statement, finishing the proof.
\end{proof}

\begin{lemma} \label{lemma:smallness of intBeps2}
Let $1<p\leq \frac{2}{2-s}$, let $D \subset \Omega$, and let $u \in C^{2}_{\beta}(D)$ with $\beta > \frac{sp}{p-1}$.
Suppose further that $B_\eps(x) \subset D$ and $x$ is such that $d_u(x)<\eps<1$. Then
\begin{align} \label{eq:intBeps bound2}
\left|{\rm P.V.}\int_{B_\eps(x)}|u(x)-u(y)|^{p-2}(u(x)-u(y))K(x,y)\,dy\right| \leq c_\eps,
\end{align}
where $c_\eps$ is independent of $x$ and $c_\eps \to 0$ as $\eps \to 0$.
\end{lemma}
\begin{proof} Again, if $\nabla u (x)=0$, the result is quite obvious since then $u\in C^{2}_{\beta}(D)$ implies
$$
|u(x)-u(y)|\leq C|x-y|^\beta, 
$$
for some constant $C$. We therefore focus on the case $\nabla u(x)\neq 0$. Let $\ell(y):=u(x)+\nabla u(x)\cdot(y-x)$ be the affine part of $u$ near $x$. Throughout the proof, $c$ will denote a constant depending on $n$, $p$, $s$, $\Lambda$, $\beta$, and $\|u\|_{C_\beta^2(D)}$. Denoting by $g(t):=|t|^{p-2}t$, we have by Lemma \ref{lemma:affines in kernel} and Lemma
\ref{lemma:simple inequality}
\begin{align*}
&\left|\int_{B_\eps(x)}|u(x)-u(y)|^{p-2}(u(x)-u(y))K(x,y)\,dy\right| \\
&\qquad \leq \int_{B_\eps(x)}\left|g(u(x)-u(y))-g(\ell(x)-\ell(y))\right|K(x,y)\,dy \\
&\qquad \leq c\int_{B_\eps(x)}\big(|\ell(x)-\ell(y)|+|u(y)-\ell(y)|\big)^{p-2}|u(y)-\ell(y)|K(x,y)\,dy \\
&\qquad = c\int_{B_\eps(x)}\big(|\nabla u(x) \cdot(y-x)|+|u(y)-\ell(y)|\big)^{p-2}|u(y)-\ell(y)|K(x,y)\,dy.
\end{align*}
After a change of variables $z=y-x$ this becomes
\begin{align*}
&\left|\int_{B_\eps(x)}|u(x)-u(y)|^{p-2}(u(x)-u(y))K(x,y)\,dy\right| \\
&\qquad \leq c\int_{B_\eps(0)}\Big(|\nabla u(x) \cdot z|+\sup_{|\xi-x|<|z|}|D^{2}u(\xi)||z|^{2}\Big)^{p-2}\!\!\!\!\sup_{|\xi-x|<|z|}|D^{2}u(\xi)||z|^{2}K(z)\,dz \\
&\qquad \leq c\int_0^{\eps}\int_{S^{n}}\bigg(|\nabla u(x) \cdot \omega|r+\sup_{B_r(x)}|D^{2}u|r^{2}\bigg)^{p-2}\sup_{B_r(x)}|D^{2}u|r^{2-n-sp+n-1}\,d\omega\,dr \\
&\qquad \leq c\int_0^{\eps}\int_{S^{n}}\left(\frac{|\nabla u(x) \cdot \omega|}{|\nabla u(x)|}+\frac{(d_u(x)+r)^{\beta-2}r}{|\nabla u(x)|}\right)^{p-2} (d_u(x)+r)^{\beta-2} \\
&\qquad \qquad\qquad\qquad \times |\nabla u(x)|^{p-2}r^{p(1-s)}\,d\omega\frac{dr}{r},
\end{align*}
where we used the monotonicity of $(a+b)^{p-2}b$ with respect to $b$ when $a,b \geq 0$, the upper bound on $K(z)$, and the upper bound on $|D^2 u|$ in $C^2_\beta$. Applying Lemma \ref{lemma:spherical estimate} and splitting the integral into two parts, we obtain
\begin{align*}
&\left|\int_{B_\eps(x)}|u(x)-u(y)|^{p-2}(u(x)-u(y))K(x,y)\,dy\right| \\
&\qquad \leq c\int_0^{d_u(x)}\left(1+\frac{(d_u(x)+r)^{\beta-2}r}{|\nabla u(x)|}\right)^{p-2} (d_u(x)+r)^{\beta-2} |\nabla u(x)|^{p-2}r^{p(1-s)}\frac{dr}{r} \\
&\qquad\quad +c\int_{d_u(x)}^{\eps}\left(1+\frac{(d_u(x)+r)^{\beta-2}r}{|\nabla u(x)|}\right)^{p-2} (d_u(x)+r)^{\beta-2} |\nabla u(x)|^{p-2}r^{p(1-s)}\frac{dr}{r} \\[1ex]
&\qquad =: I_1 + I_2.
\end{align*}
The first integral can be estimated as
\begin{align*}
I_1 &\leq c\int_0^{d_u(x)} d_u(x)^{\beta-2} |\nabla u(x)|^{p-2}r^{p(1-s)}\frac{dr}{r} \\
&\leq c\,d_u(x)^{\beta-2} d_u(x)^{(\beta-1)(p-2)}d_u(x)^{p(1-s)} \\[1ex]
&\leq c\,\eps^{\beta(p-1)-sp}
\end{align*}
by the lower bound on $|\nabla u|$ in $C^2_\beta$ and the fact $\beta>\frac{sp}{p-1}$. 
For the second integral, we have 
\begin{align*}
I_2 &\leq c\int_{d_u(x)}^{\eps}\left(\frac{r^{\beta-1}}{|\nabla u(x)|}\right)^{p-2} r^{\beta-2} |\nabla u(x)|^{p-2}r^{p(1-s)}\frac{dr}{r} \\
&= c\int_{d_u(x)}^{\eps}r^{\beta(p-1)-sp}\frac{dr}{r} \\[1ex]
&\leq c\,\eps^{\beta(p-1)-sp}
\end{align*}
since $\beta>\frac{sp}{p-1}$. Combining our estimates for $I_1$ and $I_2$, 
we obtain \eqref{eq:intBeps bound2}.
\end{proof}

\subsection{Continuity properties} \label{ss.CP}

We are now ready to prove that $\L\phi$ is continuous for appropriate $\phi$ (as in Definition \ref{def:viscosity super}).

\begin{lemma} \label{lemma:Lphi cont}
Let $B_r(x_0) \subset \Omega$ and $\phi \in C^{2}(B_r(x_0)) \cap L^{p-1}_{sp}(\ern)$.
If $1<p\leq \frac{2}{2-s}$ and $\nabla \phi(x_0)=0$, we further assume that $\phi \in C^{2}_{\beta}(B_r(x_0))$
with $\beta > \frac{sp}{p-1}$.
Then $\L \phi$ is continuous in $B_{r}(x_0)$.
\end{lemma}
\begin{proof}
Let $x \in B_{r}(x_0)$ and $\eps>0$. If $\nabla \phi(x) \neq 0$, then there is $\delta>0$ such that $\nabla \phi(y)\neq 0$ when $|x-y|\leq \delta$
by continuity. According to Lemma \ref{lemma:smallness of intBeps1}, we can then choose $\rho >0$ such that
\begin{align} \label{eq:Brho small}
\left|{\rm P.V.}\int_{B_{\rho}(y)}|\phi(y)-\phi(z)|^{p-2}(\phi(y)-\phi(z))K(y,z)\,dz\right| < \frac{\eps}{4}
\end{align}
whenever $|x-y|<\delta$. In fact, in the case $p>\frac{2}{2-s}$, we obtain \eqref{eq:Brho small} by Lemma \ref{lemma:smallness of intBeps1} regardless of the value of $\nabla \phi(x)$.
If, in turn, $1<p \leq \frac{2}{2-s}$ and $\nabla \phi(x)=0$, then $d_\phi(y) < \rho$ when $|x-y|<\rho$, and we have \eqref{eq:Brho small} in this case, as well, by Lemma \ref{lemma:smallness of intBeps2}.

Let us then consider the nonlocal contribution.
We may assume $|x-y|<\rho/3$.
Then we can estimate
\begin{align*}
&\chi_{\ern \setminus B_{\rho}(y)}(z)|\phi(y)-\phi(z)|^{p-1}K(y,z) \\[2ex]
&\qquad \leq c\,\chi_{\ern \setminus B_{\rho}(y)}(z)\left(|\phi(y)|^{p-1}+|\phi(z)|^{p-1}\right)|y-z|^{-n-sp} \\
&\qquad \leq c\,\chi_{\ern \setminus B_{2\rho/3}(x)}(z)\left(\sup_{B_{\rho/3}(x)}|\phi|^{p-1}+|\phi(z)|^{p-1}\right)|x-z|^{-n-sp},
\end{align*}
and consequently
\begin{align} \label{eq:tail convergence}
&\int_{\ern \setminus B_{\rho}(y)}|\phi(y)-\phi(z)|^{p-2}(\phi(y)-\phi(z))K(y,z)\,dz \nonumber \\
&\qquad \to \int_{\ern \setminus B_{\rho}(x)}|\phi(x)-\phi(z)|^{p-2}(\phi(x)-\phi(z))K(x,z)\,dz
\end{align}
as $y \to x$ by the dominated convergence theorem together with the assumption $\phi \in L^{p-1}_{sp}(\ern)$ and continuity of $K(\cdot,z)$ away from the diagonal.
Combining \eqref{eq:Brho small} with \eqref{eq:tail convergence}, we obtain
\[
|\L \phi(x)-\mathcal \L \phi(y)| < \eps
\]
when $|x-y|$ is small enough, and the proof is complete.
\end{proof}

The next lemma states that $\L$ is continuous with respect to perturbations that are regular enough.

\begin{lemma} \label{lemma:Lphi cont in eps}
Let $B_{r}(x_0) \subset \Omega$ and let $\phi \in C^{2}(B_r(x_0)) \cap L^{p-1}_{sp}(\ern)$
satisfy Definition \ref{def:viscosity super}(iii) (a) or (b) with  $\beta > \frac{sp}{p-1}$.
Then for every $\eps>0$ and $\rho'>0$ there exist $\theta'>0$, $\rho \in (0,\rho')$ and $\eta \in C^{2}_0(B_{\rho/2}(x_0))$ with $0 \leq \eta \leq 1$ and $\eta(x_0)=1$
such that $\phi_\theta := \phi + \theta \eta$ satisfies
\begin{align*}
\sup_{B_{\rho}(x_0)}|\L \phi-\L \phi_\theta| < \eps
\end{align*}
whenever $0\leq \theta<\theta'$.
\end{lemma}
\begin{proof}
Let $\eps>0$ and $\rho'>0$. Firstly, if $\nabla \phi(x_0) \neq 0$, there exist $\rho \in (0,\rho')$ and $\tau>0$ such that $|\nabla \phi| > \tau$ in $B_{2\rho}(x_0)$ by continuity. Letting $\eta \in C^{2}_0(B_{\rho/2}(x_0))$ satisfy $0 \leq \eta \leq 1$ and $\eta(x_0)=1$,
we then have $|\nabla \phi_\theta|>\tau/2$ in $B_{2\rho}(x_0)$ when $0\leq\theta<\theta''$ for some $\theta''>0$. According to Lemma \ref{lemma:smallness of intBeps1}, we can now take such a small $\delta>0$ that for every $x \in B_{\rho}(x_0)$
and every $\theta$ as above it holds
\begin{align} \label{eq:intB phitheta small}
\left|{\rm P.V.}\int_{B_{\delta}(x)}|\phi_\theta(x)-\phi_\theta(y)|^{p-2}(\phi_\theta(x)-\phi_\theta(y))K(x,y)\,dy\right| < \frac{\eps}{4}.
\end{align}
If $p>\frac{2}{2-s}$, we obtain \eqref{eq:intB phitheta small} by Lemma \ref{lemma:smallness of intBeps1} regardless of the value of $\nabla \phi(x_0)$.

Let us then consider the case $1<p\leq \frac{2}{2-s}$ and $\nabla \phi(x_0)=0$. Since $x_0$ is an isolated critical point of $\phi$ by assumption, we can also assume that $\rho$ is chosen so small that $|\nabla \phi| \neq 0$ in $B_{3\rho}(x_0) \setminus \{x_0\}$. Let $\eta \in C^{2}_0(B_{\rho/2}(x_0))$ satisfy $0 \leq \eta \leq 1$, $\eta=1$ in $B_{\rho/4}(x_0)$ and $|D^{2}\eta|\leq M d_\eta^{\beta-2}$ for some constant $M>0$. Then, in particular, $\nabla \phi_\theta \neq 0$ in $B_{2\rho}(x_0) \setminus \{x_0\}$ when $\theta$ is small enough, and consequently $d_{\phi_\theta} = d_{\phi}$ in $B_\rho(x_0)$ for all such $\theta$. Also, $\frac12 |\nabla \phi| \leq |\nabla \phi_\theta| \leq 2|\nabla \phi|$ in $B_\rho(x_0)$ when $\theta$ is small enough.
Moreover, we can estimate
\[
|D^{2}\phi_\theta| \leq |D^{2}\phi|+\theta|D^{2}\eta| \leq \| \phi \|_{C_\beta^2(B_\rho(x_0))} d_\phi^{\beta-2}+\theta M d_\eta^{\beta-2} \leq c\,d_{\phi_\theta}^{\beta-2}
\]
in $B_\rho(x_0)$ whenever $\theta$ is small enough, since $d_\eta \leq d_\phi = d_{\phi_\theta}$ in $B_\rho(x_0)$.
Thus $\phi_\theta \in C_{\beta}^2(B_{\rho}(x_0))$,  and according to Lemma \ref{lemma:smallness of intBeps2}, we find $\delta \in(0,\rho)$ such that~\eqref{eq:intB phitheta small} holds also in this case.

Letting now $x \in B_{\rho}(x_0)$ and denoting by $g(t):=|t|^{p-2}t$, we can estimate by \eqref{eq:intB phitheta small} and Lemma \ref{lemma:simple inequality} as
\begin{align*}
&\left|\L \phi(x)-\L \phi_\theta(x)\right| \\
&\qquad \leq \frac{\eps}{2} +\int_{\ern \setminus B_{\delta}(x)} \left|g(\phi(x)-\phi(y))-g(\phi_\theta(x)-\phi_\theta(y))\right|K(x,y)\,dy \\
&\qquad \leq \frac{\eps}{2} +c\int_{\ern \setminus B_{\delta}(x)} 2\theta\big(|\phi(x)-\phi(y)|+2\theta\big)^{p-2} |x-y|^{-n-sp}\,dy,
\end{align*}
where we also used monotonicity of $(a+b)^{p-2}b$ with respect to $b$ when $a,b \geq 0$ together with
\[
|\phi(x)-\phi(y)-\phi_\theta(x)+\phi_\theta(y)| \leq |\phi(x)-\phi_\theta(x)|+|\phi(y)-\phi_\theta(y)| \leq 2\theta.
\]
If $1<p<2$, we can simply continue estimating
\begin{align*}
\left|\L \phi(x)-\L \phi_\theta(x)\right|
&\leq \frac{\eps}{2}+c\,\theta^{p-1}\int_{\ern \setminus B_{\delta}(x)} |x-y|^{-n-sp}\,dy \\
&\leq \frac{\eps}{2}+c\,\delta^{-sp}\theta^{p-1}<\eps
\end{align*}
when $\theta$ is small enough. If $p\geq 2$, in turn, we obtain
\begin{align*}
&\left|\L \phi(x)-\L \phi_\theta(x)\right| \\
&\qquad \leq \frac{\eps}{2}+c\int_{\ern \setminus B_{\delta}(x)}\theta\left(\theta^{p-2}+|\phi(x)|^{p-2}+|\phi(y)|^{p-2}\right)|x-y|^{-n-sp}\,dy \\
&\qquad \leq \frac{\eps}{2}+c\,\delta^{-sp}\theta^{p-1}+c\,\delta^{-sp}\theta \sup_{B_{\rho}(x_0)}|\phi|^{p-2} +c\,\delta^{-sp}\theta\sup_{\xi \in B_{\rho}(x_0)}{\rm Tail}(\phi;\xi,\delta)^{p-2},
\end{align*}
where we used H\"older's inequality to estimate
\begin{align*}
&\int_{\ern \setminus B_{\delta}(x)}|\phi(y)|^{p-2}|x-y|^{-n-sp}\,dy \\
&\qquad \leq \left(\int_{\ern \setminus B_{\delta}(x)}|x-y|^{-n-sp}\,dy\right)^{\frac{1}{p-1}}
\left(\int_{\ern \setminus B_{\delta}(x)}|\phi(y)|^{p-1}|x-y|^{-n-sp}\,dy\right)^{\frac{p-2}{p-1}} \\
&\qquad \leq c\,\delta^{-sp/(p-1)} \delta^{-sp(p-2)/(p-1)} {\rm Tail}(\phi;x,\delta)^{p-2} \\[2ex]
&\qquad \leq c\,\delta^{-sp}\sup_{\xi \in B_{\rho}(x_0)}{\rm Tail}(\phi;\xi,\delta)^{p-2}.
\end{align*}
Thus we get
\[
\left|\L \phi(x)-\L \phi_\theta(x)\right| < \eps
\]
in this case, as well, whenever $\theta$ is small enough. The claim follows by taking the supremum over $x \in B_{\rho}(x_0)$.
\end{proof}

The following lemma establishes the expected result that any $C^2$-supersolution is also a weak supersolution.
\begin{lemma} \label{lemma:super weak super} Let $u \in C^{2}(B_r(x_0)) \cap L^{p-1}_{sp}(\ern)$ and if $1<p\leq \frac{2}{2-s}$ and $\nabla u(x_0)=0$, we further assume that $u \in C^{2}_{\beta}(B_r(x_0))$
with  $\beta > \frac{sp}{p-1}$. If $\mathcal \L u \geq 0$ in the pointwise sense in $B_r(x_0)$, then $u$ is a continuous weak supersolution in $B_r(x_0)$.
\end{lemma}
\begin{proof}
Clearly $u \in W^{s,p}_{\rm loc}(B_r(x_0))$. Let $\phi \in C^{\infty}_0(B_r(x_0))$ be a nonnegative test function.
Since $\L u \geq 0$, we have by the definition of $\L$ that
\[
\int_{\ern \setminus B_\eps(x)}|u(x)-u(y)|^{p-2}(u(x)-u(y))K(x,y)\,dy \geq -\delta_\eps(x), \qquad x \in \supp\phi,
\]
for every $\eps>0$, where $\delta_\eps(x) \to 0$ uniformly as $\eps \to 0$ due to the continuity of $\L u$, i.e. Lemma \ref{lemma:Lphi cont}.
Multiplying the above inequality by $\phi(x)$ and integrating over $\ern$, we obtain
\begin{align}
&\int_{\ern} \int_{\ern}\big(1-\chi_{B_\eps(x)}(y)\big)|u(x)-u(y)|^{p-2}(u(x)-u(y))\phi(x)K(x,y)\,dydx \nonumber \\
&\qquad \geq -\int_{\ern}\delta_\eps(x)\phi(x)\,dx. \label{eq:s1}
\end{align}
Interchanging the roles of $x$ and $y$ yields
\begin{align}
&\int_{\ern} \int_{\ern}\big(1-\chi_{B_\eps(y)}(x)\big)|u(x)-u(y)|^{p-2}(u(y)-u(x))\phi(y)K(x,y)\,dxdy \nonumber \\
&\qquad \geq -\int_{\ern}\delta_\eps(x)\phi(x)\,dx\label{eq:s2}
\end{align}
by the symmetry of $K$. Summing up \eqref{eq:s1} and \eqref{eq:s2} and changing order of integration in the first one, we obtain
\[
\int_{\ern} \int_{\ern \setminus B_\eps(y)}|u(x)-u(y)|^{p-2}(u(x)-u(y))(\phi(x)-\phi(y))K(x,y)\,dxdy \geq -2\|\delta_\eps\phi\|_{L^{1}}
\]
for every $\eps>0$. Letting now $\eps \to 0$, we have
\[
\int_{\ern} \int_{\ern}|u(x)-u(y)|^{p-2}(u(x)-u(y))(\phi(x)-\phi(y))K(x,y)\,dxdy \geq 0
\]
by the dominated convergence theorem.
To see that the integrand has an integrable upper bound, let $\supp \phi \subset B_\rho \Subset B_r(x_0)$ and estimate by H\"older's inequality
\begin{align*}
&\int_{\ern} \int_{\ern}|u(x)-u(y)|^{p-1}|\phi(x)-\phi(y)|K(x,y)\,dxdy \\
&\qquad \leq c\int_{B_\rho} \int_{B_\rho}|u(x)-u(y)|^{p-1}|\phi(x)-\phi(y)|\frac{dxdy}{|x-y|^{n+sp}} \\
&\qquad\quad + c\int_{\ern \setminus B_\rho} \int_{\supp\phi}|u(x)-u(y)|^{p-1}\phi(x)|x-y|^{-n-sp}\,dxdy \\
&\qquad \leq c\left(\int_{B_\rho} \int_{B_\rho}\frac{|u(x)-u(y)|^{p}}{|x-y|^{n+sp}}\,dxdy\right)^{\frac{p-1}{p}}
\left(\int_{B_\rho} \int_{B_\rho}\frac{|\phi(x)-\phi(y)|^{p}}{|x-y|^{n+sp}}\,dxdy\right)^{\frac1p} \\
&\qquad\quad + c\int_{\ern \setminus B_d(z)} \int_{\supp\phi}\left(|u(x)|^{p-1}+|u(y)|^{p-1}\right)\phi(x)|z-y|^{-n-sp}\,dxdy \\
&\qquad \leq c\,\|u\|_{W^{s,p}(B_\rho)}^{p-1}\|\phi\|_{W^{s,p}(B_\rho)}+c\,\|\phi\|_{L^{1}(B_\rho)}{\rm Tail}(u;z,d)^{p-1} \\[1ex]
&\qquad < \infty,
\end{align*}
where $z \in \supp\phi$ and $d:=\dist(z,\partial B_\rho)$. We conclude that $u$ is a weak supersolution in $B_r(x_0)$.
\end{proof}

Finally, we conclude this section with the result saying that whenever we can touch an $(s,p)$-viscosity supersolution from below with a $C^2$-function, then the principal value is well-defined  and nonnegative at that touching point.
\begin{proposition} \label{prop:principal value}
Let $u$ be an $(s,p)$-viscosity supersolution in $\Omega$. Assume that there is a $C^2$-function $\phi$
touching $u$ from below at $x \in \Omega$, i.e., there is $r>0$ such that
\[
\phi(x) = u(x) \qquad \mbox{and} \qquad \phi \leq u \,\text{ in } B_r(x) \subset \Omega.
\]
If $\phi$ satisfies (a) or (b) with $\beta > \frac{sp}{p-1}$  in Definition \ref{def:viscosity super}(iii), then the principal value $\L u(x)$ exists and is nonnegative.
\end{proposition}

\begin{proof}
Without loss of generality we may assume that $x = 0$ and $u(0)=0$. For $\rho \in (0,r)$, define  
\[
\phi_\rho(y) := \begin{cases}
 \phi(y), & y \in B_\rho, \\
  u(y), & y \in \ern \setminus B_\rho.
\end{cases}
\]
First, we show that the principal value exists. 
Setting $K(y) = K(0,y)$, we have
\begin{align*} 
\int_{B_\rho \setminus B_\delta} |u(y)|^{p-2} u(y) K(y) \, dy
&= \int_{B_\rho \setminus B_\delta} \left(|u(y)|^{p-2} u(y)-|\phi(y)|^{p-2}\phi(y)\right) K(y) \, dy \\
&\quad + \int_{B_\rho \setminus B_\delta} |\phi(y)|^{p-2}\phi(y) K(y) \, dy \\[1ex]
& =: I_{1,\delta}+I_{2,\delta}
\end{align*}
whenever $0<\delta<\rho<r$. Since $u\geq\phi$ in $B_\rho$, the integrand of $I_{1,\delta}$ is nonnegative and the limit $\lim_{\delta \to 0}I_{1,\delta}$ exists by the monotone convergence theorem. For $I_{2,\delta}$, in turn, the limit $\lim_{\delta \to 0}I_{2,\delta}$ exists
by Definition \ref{def:viscosity super}(iii). 
In addition,
\[
\int_{\ern \setminus B_\rho} |u(y)|^{p-2} u(y) K(y) \, dy > -\infty
\]
by the fact $u_- \in L^{p-1}_{sp}(\ern)$, and thus the principal value $\L u(0)$ exists.

Let us then show that $\L u(0) \geq 0$. Let $\eps>0$. By Lemma \ref{lemma:smallness of intBeps1} and Lemma \ref{lemma:smallness of intBeps2} we can take $\rho$ to be so small that
\[
\left| {\rm P.V.}\int_{B_\rho} |\phi(y)|^{p-2} \phi(y) K(y) \, dy \right| < \eps.
\]
Consequently, we can estimate
\begin{align*}
&\int_{\ern \setminus B_\rho} |u(y)|^{p-2} (-u(y)) K(y) \, dy = \int_{\ern \setminus B_\rho} |\phi_\rho(y)|^{p-2} (-\phi_\rho(y)) K(y) \, dy \\
&\qquad = {\rm P.V.}\int_{\ern} |\phi_\rho(y)|^{p-2} (-\phi_\rho(y)) K(y) \, dy - {\rm P.V.}\int_{B_\rho} |\phi(y)|^{p-2} (-\phi(y)) K(y) \, dy \\
&\qquad \geq \L \phi_\rho(0)- \eps \geq -\eps
\end{align*}
by Definition \ref{def:viscosity super}(iii). Hence letting $\rho \to 0$ yields $\L u(0) \geq -\eps$, and the claim follows by letting $\eps \to 0$.
\end{proof}

\section{Comparison principle} \label{s.comp}

In this section, we prove a weak comparison principle for viscosity solutions. This is one of the keys to our main result. First, a small lemma related to integrable functions is stated and proved.
\begin{lemma} \label{lemma:ae convergence uxz}
Let $p>1$ and let $u$ be a measurable function 
with $u_- \in L^{p-1}_{\rm loc}(\ern)$.
Let $\{x_k\}$ be a sequence in $\ern$ converging to $x \in \ern$. Then there exists a subsequence $\{x_{k_j}\}_j$ such that
\begin{align} \label{eq:ae convergence uxz}
\liminf_{j \to \infty} u(x_{k_j}+z) \geq u(x+z) \quad \text{for a.e. } z \in \ern.
\end{align}
\end{lemma}
\begin{proof}
For every positive integer $m$, denote by $u_m := \min\{u,m\} \in L^{p-1}_{\rm loc}(\ern)$ and
let $\{K_m\}$ be a compact exhaustion of $\ern$.
Then, it is a well known fact that for each $m$
\[
\lim_{k\to\infty} \int_{K_m} |u_m(x_k+z)-u_m(x+z)|^{p-1}\,dz = 0.
\]
Hence, using a diagonal argument, we can extract a subsequence $\{x_{k_j}\}$ such that for every $m$
\begin{align*} \label{eq:ae convergence uxz2}
\lim_{j \to \infty} u_m(x_{k_j}+z) = u_m(x+z) \quad \text{for a.e. } z \in K_m.
\end{align*}
Now we can estimate
\[
\liminf_{j \to \infty}u(x_{k_j}+z) \geq \liminf_{j \to \infty}u_m(x_{k_j}+z)=u_m(x+z)
\]
for almost every $z\in K_m$, and finally \eqref{eq:ae convergence uxz} follows by letting $m \to \infty$. 
\end{proof}


\begin{theorem}[{\bf Comparison principle of viscosity solutions}] \label{thm:comparison}
Let $u$ and $v$ be an $(s,p)$-viscosity supersolution and an $(s,p)$-viscosity subsolution, respectively, in $\Omega$. Assume further that both $v$ and $-u$ are upper semicontinuous in $\overline \Omega$ and $u \geq v$ on $\partial\Omega$ and almost everywhere in $\ern \setminus \Omega$. Then $u \geq v$ in $\Omega$. 
\end{theorem}

\begin{proof}
Assume contrary to the claim that there is a point $x_0 \in \Omega$ such that 
\[
\sigma := \sup_{\Omega} (v-u) = v(x_0) - u(x_0)>0.
\]
Set
\[
H := \sup_{\Omega}v  - \inf_{\Omega}u,
\] 
which is a finite number by the assumed semicontinuity properties. Thus also $\sigma$ is finite. Define 
\[
\Psi_\eps(x,y) := v(x) - u(y) - \frac{1}{\eps} |x-y|^q,
\]
where $q=2$ if  $p>\frac{2}{2-s}$ and $q > \frac{sp}{p-1}$ otherwise, and let 
\[
M_\eps := \sup_{x,y \in  \Omega} \Psi_\eps(x,y).
\]
Clearly $M_\eps \leq H$ and $M_\eps \geq \Psi_\eps(x_0,x_0) = \sigma$. Moreover, since $\Psi_{\eps_1}(x,y) \leq \Psi_{\eps_2}(x,y)$ whenever $\eps_1 \leq \eps_2$, we see that $M_{\eps_1} \leq M_{\eps_2}$. 
Therefore, we have that $M := \lim_{\eps \to 0} M_\eps$ exists by the uniform lower bound $M_\eps \geq \sigma$. Furthermore, by the upper semicontinuity of $v$ and $-u$, for any $\eps>0$ there are points $x_\eps,y_\eps \in \overline \Omega$  such that 
\[
M_\eps = \Psi_\eps(x_\eps,y_\eps).
\]

Let us now analyze the limit. Firstly, we see that 
\[
M_{2\eps} \geq \Psi_{2\eps}(x_{\eps},y_{\eps}) = M_{\eps} + \frac{1}{2\eps} |x_\eps-y_\eps|^q,
\]
and thus 
\begin{equation}
\label{eq:Meps}
\frac{1}{\eps} |x_\eps-y_\eps|^q \leq 2(M_{2\eps}-M_{\eps}) \to 0
\end{equation}
as $\eps \to 0$. Secondly, let $x^*\in \overline \Omega$ be any accumulation point of $\{x_\eps\}$.  Then there is a subsequence $\{x_{\eps_j}\}_j$ such that $x_{\eps_j} \to x^*$ as $j \to \infty$. Also $y_{\eps_j} \to x^*$ as $j \to \infty$ by \eqref{eq:Meps}. The upper semicontinuity of $v$ and $-u$ then implies
\begin{align*} \label{eq:} 
\sigma &\leq \lim_{j \to \infty} M_{\eps_j} = \lim_{j \to \infty} \big(v(x_{\eps_j}) - u(y_{\eps_j})\big)
-\lim_{j \to \infty} \frac{1}{\eps_j} |x_{\eps_j}-y_{\eps_j}|^q \\
&\leq  \limsup_{j \to \infty} \big(v(x_{\eps_j}) - u(y_{\eps_j})\big) \leq v(x^*) - u(x^*) \leq \sigma.
\end{align*}
Therefore $x^*$ must be in $\Omega$, because otherwise the boundary condition would be violated. Upon relabeling the subsequence $\{x_{\eps_j}\}$ as $\{x_\eps\}$, we thus have
\[
\lim_{\eps \to 0} \big(v(x_{\eps}) - u(y_{\eps})\big)  = v(x^*) - u(x^*) = \sup_{\Omega} (v-u).
\]

From now on, we assume that $\eps$ is so small that $x_\eps,y_\eps \in B_{r}(x^*)$ for $r:= \frac13 \dist(x^*,\partial \Omega)$. We introduce the set
\[
E_{y} := \{ z \in \ern :  y + z \in \Omega \}, \qquad y \in \Omega.
\]
Let us continue with further consequences of the definitions above. Since 
\[
\Psi_\eps(x_\eps,y_\eps) \geq \Psi_\eps(x_\eps+z,y_\eps + z),
\]
we obtain
\[
v(x_\eps) - u(y_\eps) - \frac{1}{\eps} |x_\eps-y_\eps|^q  \geq v(x_\eps+z) - u(y_\eps+z) - \frac{1}{\eps} |x_\eps-y_\eps|^q 
\]
for all $z \in E_{x_\eps} \cap E_{y_\eps}$. In particular,
\begin{equation} \label{eq:W_eps 1}
W_\eps(z) := v(x_\eps) - v(x_\eps+z) - u(y_\eps) + u(y_\eps+z) \geq 0 
\end{equation}
for all such $z$. Moreover, as 
\[
\Psi_\eps(x_\eps,y_\eps) \geq \Psi_\eps(x_\eps + z,y_\eps) \qquad \mbox{and} \qquad \Psi_\eps(x_\eps,y_\eps) \geq \Psi_\eps(x_\eps,y_\eps+z) 
\]
hold for all $z \in B_{2r}(0)$, we also see that 
\begin{align*} 
v(x_\eps + z) &\leq v(x_\eps) - \frac{1}{\eps} |x_\eps-y_\eps|^q + \frac{1}{\eps} |x_\eps+z-y_\eps|^q
\end{align*}
and
\begin{align*} 
u(y_\eps + z) &\geq u(y_\eps) + \frac{1}{\eps} |x_\eps-y_\eps|^q - \frac{1}{\eps} |x_\eps-y_\eps-z|^q 
\end{align*}
for all such $z$. Thus, there are $C^2$-functions
\[
\phi_\eps(x):=v(x_\eps)-\frac1\eps|x_\eps-y_\eps|^{q}+\frac1\eps|x-y_\eps|^{q}
\]
and
\[
\psi_\eps(y):=u(y_\eps)+\frac1\eps|x_\eps-y_\eps|^{q}-\frac1\eps|x_\eps-y|^{q}
\]
touching $v$ from above at $x_\eps$ and $u$ from below at $y_\eps$, respectively. In addition, if $\nabla \phi_\eps(x_\eps)=0$ or $\nabla \psi_\eps(y_\eps)=0$, then $x_\eps=y_\eps$ and it is an isolated critical point for both $\phi_\e$ and $\psi_\e$.
Moreover, if $1<p\leq \frac{2}{2-s}$, then clearly $\phi_\e,\psi_\e \in C^{2}_{q}(\Omega)$.

Since $v$ and $u$ are an $(s,p)$-viscosity subsolution and a supersolution, respectively, we have from Proposition \ref{prop:principal value} that $\mathcal{L} v(x_\eps) \leq 0$ and $\mathcal{L} u(y_\eps) \geq 0$ in the pointwise sense.
Furthermore, using the translation invariance of $K$ and performing a change of variables $z=x-x_\eps$ we get
\begin{align*}
0 \geq \mathcal{L} v(x_\eps) &= \int_{\ern} |v(x_\eps) - v(x)|^{p-2} (v(x_\eps) - v(x)) K(x_\eps,x) \, dx \\
&= \int_{\ern} |v(x_\eps) - v(x_\eps+z)|^{p-2} \big(v(x_\eps) - v(x_\eps+z)\big) K(z,0) \, dz,
\end{align*}
and, similarly,
\begin{align*}
0 \leq \mathcal{L} u(y_\eps) &= \int_{\ern} |u(y_\eps) - u(x)|^{p-2} (u(y_\eps) - u(x)) K(y_\eps,x) \, dx \\
&= \int_{\ern} |u(y_\eps) - u(y_\eps+z)|^{p-2} \big(u(y_\eps) - u(y_\eps+z)\big) K(z,0) \, dz.
\end{align*}
Therefore,
\begin{align}  \label{thetae bound}
0 \geq  \mathcal{L} v(x_\eps) - \mathcal{L} u(y_\eps) = \int_{\ern} \Theta_\eps(z)  \, d\nu(z),
\end{align}
where
\begin{align*} 
\Theta_\eps(z) := \,&|v(x_\eps) - v(x_\eps+z)|^{p-2} \big(v(x_\eps) - v(x_\eps+z)\big) \\
& - |u(y_\eps) - u(y_\eps+z)|^{p-2} \big(u(y_\eps) - u(y_\eps+z)\big) \nonumber
\end{align*}
and
\begin{align*}
d\nu(z) := K(z,0)\,dz.
\end{align*}

Decompose now $\ern$ as
\begin{align*}
\ern = (E_{x_\eps} \cap E_{y_\eps}) \cup \big(\R^n\setminus ( E_{x_\eps} \cap E_{y_\eps})\big) 
=: E^{1,\eps} \cup E^{2,\eps}.
\end{align*}
Straightforward manipulations show that 
\begin{align} \label{Thetaez2}
\Theta_\eps(z) &= (p-1)\int_0^1 \left| t\big(v(x_\eps) - v(x_\eps+z)\big) + (1-t)\big(u(y_\eps) - u(y_\eps+z)\big)\right|^{p-2}  \,dt \nonumber \\
&\qquad \qquad \times \big(v(x_\eps) - v(x_\eps+z) - u(y_\eps) + u(y_\eps+z)\big)  \\
&= (p-1)\int_0^1 | u(y_\eps) - u(y_\eps+z) + t W_\eps(z)|^{p-2}  \, dt \, W_\eps(z)\nonumber
\end{align}
with $W_\eps$ defined in \eqref{eq:W_eps 1}.
Thus, due to Lemma \ref{lemma:stupid integral 2} and the nonnegativity of $W_\eps(z)$ whenever $z \in E^{1,\eps}$,
\[
\Theta_\eps(z) \geq \frac1c \big(| u(y_\eps) - u(y_\eps+z)| + W_\eps(z) \big)^{p-2}W_\eps(z), 
\] 
for all $z \in E^{1,\eps}$. Therefore 
\begin{align} \label{E1e bound}
\nonumber \liminf_{\eps \to 0} &\int_{E^{1,\eps}} \Theta_\eps \, d\nu  \\
\geq &\frac1c \liminf_{\eps \to 0}  \int_{E^{1,\eps}} 
 \big(| u(y_\eps) - u(y_\eps+z)| + W_\eps(z) \big)^{p-2}W_\eps(z)  \, d\nu(z).
\end{align}

Now consider the set $E^{2,\eps}$. Since $u_-, v_+ \in L^{p-1}_{sp}(\ern)$, we can by Lemma \ref{lemma:ae convergence uxz}, upon extracting a subsequence, assume
\begin{align} \label{eq:E3 convergence}
\liminf_{\eps \to 0} W_\eps(z) \geq \sigma + u(x^*+z)- v(x^*+z) 
\end{align}
for almost every $z$ due to the pointwise convergence of $v(x_\e)-u(y_\e)$.
By picking yet another subsequence, we can also assume
$$
\chi_{E^{2,\eps}}\to \chi_{\R^n\setminus E_{x^*}}
$$
almost everywhere as $\eps \to 0$,
and thus by the order of boundary values for $u$ and $v$
\begin{align} \label{Wnonneg}
\liminf_{\eps \to 0} \big[W_\eps(z)\chi_{E^{2,\eps}}(z)\big] &\geq \big(\sigma + u(x^*+z)- v(x^*+z)\big)\chi_{\R^n\setminus E_{x^*}}(z) \nonumber \\
&\geq \sigma\,\chi_{\R^n\setminus E_{x^*}}\!(z)
\end{align}
for almost every $z$. Since, in addition, $|z|^{-1}$ is bounded in $E^{2,\eps}$ and $u_-, v_+ \in L^{p-1}_{sp}(\ern)$, it is now easy to see that $\Theta_\e\chi_{E^{2,\eps}}$ is bounded from below by a uniformly integrable function. Indeed, for the part involving $v$ we have
\begin{align*}
|v(x_\e)-v(x_\e+z)|^{p-2}(v(x_\e)-v(x_\e+z)) \geq -c\,\big(|v(x_\e)|^{p-1}+|v_+(x_\e+z)|^{p-1}\big),
\end{align*}
and
\begin{align*}
\int_{E^{2,\eps}} \big(|v(x_\e)|^{p-1} + |v_+(x_\e {+} z)|^{p-1}\big)\,d\nu(z)
&\leq c\,r^{-sp}\left(|v(x_\eps)|^{p-1} + {\rm Tail}(v_+;x_\eps,2r)^{p-1}\right) \\
&\leq c\,r^{-sp}\left(|v(x_\eps)|^{p-1} + {\rm Tail}(v_+;x^*,r)^{p-1}\right),
\end{align*}
where we have used that $|z|>2r$ in $E^{2,\eps}$, since  $x_\e,y_\e \in B_r(x^*)$. The second term can easily be seen to be uniformly bounded since $v_+ \in L^{p-1}_{sp}(\ern)$. The first term is uniformly bounded by the fact $v(x^*)-u(x^*)=\sigma$ together with the semicontinuity and finiteness of $v$ and $u$ in $\Omega$.
The part involving $u$ can be treated similarly.

By \eqref{Wnonneg}, $\chi_{\{W_\e<0\}}\chi_{E^{2,\e}}\to 0$ almost everywhere, so that the dominated convergence theorem implies
$$
\lim_{\eps \to 0}\int_{E^{2,\e}} \Theta_\eps \chi_{\{W_\e<0\}} \, d\nu =0.
$$
By  \eqref{Thetaez2}, Lemma \ref{lemma:stupid integral 2}, and Fatou's lemma, in turn,
\begin{align*}
&\liminf_{\eps \to 0}\int_{E^{2,\e}} \Theta_\eps\chi_{\{W_\e\geq 0\}}  \, d\nu \\
&\qquad \geq \frac1c \liminf_{\eps \to 0}  \int_{E^{2,\eps}} 
 \big(| u(y_\eps) - u(y_\eps+z)| + W_\eps(z) \big)^{p-2}W_\eps(z)\chi_{\{W_\e\geq 0\}}  \, d\nu(z)\geq 0.
\end{align*}
Hence, we conclude
\begin{align} \label{E2e bound}
\liminf_{\eps \to 0}\int_{E^{2,\e}} \Theta_\eps  \, d\nu\geq 0.
\end{align}

Combining inequalities \eqref{thetae bound}, \eqref{E1e bound}, and \eqref{E2e bound}, we see by Fatou's lemma together with the nonnegativity of $W_\eps$ in $E^{1,\eps}$ that 
\[
 \int_{\ern} \liminf_{\eps \to 0} \left[\chi_{E^{1,\eps}}(z) \big(| u(y_\eps) - u(y_\eps+z)| + W_\eps(z) \big)^{p-2}W_\eps(z) \right] d\nu(z) \leq 0.
\]
But this is only possible if 
\begin{align*}
0 &= \liminf_{\eps \to 0} W_\eps(z) = v(x^*) -u(x^*) - \limsup_{\eps \to 0}\big(v(x_\eps+z) -u(y_\eps+z)\big) \\
&\geq \sigma - \big(v(x^*+z) -u(x^*+z)\big)
\end{align*}
for almost every $z \in E_{x^*}$ by the upper semicontinuity of $v$ and $-u$ in $\overline \Omega$. Thus $v-u \geq \sigma$ almost everywhere in $\Omega$. The upper semicontinuity of $v-u$ in $\overline \Omega$ then implies for $x \in \partial \Omega$ that 
\[
0 \geq v(x) - u(x) \geq \limsup_{y \to x} \big(v(y) - u(y)\big) \geq \sigma>0;
\]
a contradiction. This finishes the proof.
\end{proof}

\section{Proof of the main result} \label{s.main}

In this section, we prove our main result which states that the notions of $(s,p)$-superharmonic functions and $(s,p)$-viscosity supersolutions coincide.
First, we will need the following two lemmas.
\begin{lemma}\label{lem:minsup} Suppose $u$ and $v$ are $(s,p)$-superharmonic in $\Omega$. Then the function $w=\min\{u,v\}$ is also $(s,p)$-superharmonic in $\Omega$.
\end{lemma}
The proof of this is obvious from the definition.
\begin{lemma}\label{lem:um} Suppose that $u$ is finite almost everywhere in $\Omega$ 
and that $u_M=\min\{u,M\}$ is an $(s,p)$-viscosity supersolution in $\Omega$ for each $M>0$. Then $u$ itself is an $(s,p)$-viscosity supersolution in $\Omega$.
\end{lemma}
\begin{proof} Clearly $u$ satisfies (i), (ii), and (iv) in Definition \ref{def:viscosity super}, and thus we only need to verify property (iii). 
To this end, take $\phi\in C^2(B_r(x_0))$ such that $\phi(x_0)=u(x_0)$, $\phi\leq u$ in $B_r(x_0) \subset \Omega$, and either (a) or (b) with $\beta > \frac{sp}{p-1}$ in Definition \ref{def:viscosity super}(iii) holds. Let $\rho \in (0,r)$. Since $\phi(x_0)=u(x_0)$ and $\phi$ is continuous, $\phi\leq M$ in $B_\rho(x_0)$ for $M$ large enough. 
Consequently, $\phi(x_0)=u_M(x_0)$ and $\phi\leq u_M$ in $B_\rho(x_0)$. Since $u_M$ is an $(s,p)$-viscosity supersolution, $\L \phi_{\rho,M}(x_0)\geq 0$ with 
$$
\phi_{\rho,M}=\begin{cases}
\phi & \text{ in }B_\rho(x_0),\\
u_M &\text{ in }\ern\setminus B_\rho(x_0).
\end{cases}
$$
Furthermore,
\begin{align*}
\L \phi_{\rho,M}(x_0)&= \int_{B_\rho(x_0)} |\phi(x_0)-\phi(y)|^{p-2}(\phi(x_0)-\phi(y)) K(x_0,y)\,dy \\
&\quad +\int_{\ern \setminus B_\rho(x_0)} |\phi(x_0)-u_M(y)|^{p-2}(\phi(x_0)-u_M(y)) K(x_0,y)\,dy,
\end{align*}
where the second integrand is uniformly bounded in $M$ by an integrable function in the set $\{u \leq \phi(x_0)\}$ since $u_-\in L^{p-1}_{sp}(\ern)$
and, on the other hand, monotone in $M$ in the set $\{u > \phi(x_0)\}$.
By the dominated and monotone convergence theorems, we can thus let $M\to \infty$ and obtain
$$
\L \phi_{\rho} (x_0)\geq 0, \quad
\phi_{\rho}=\begin{cases}
\phi & \text{ in }B_\rho(x_0),\\
u &\text{ in }\ern\setminus B_\rho(x_0).
\end{cases}
$$
Finally, letting $\rho \to r$, we get the desired conclusion.
\end{proof}

\begin{proof}[~Proof of Theorem \ref{thm:superharmvisc}]
Assume first that $u$ is $(s,p)$-superharmonic in $\Omega$.
The only property left to verify is property (iii) in Definition \ref{def:viscosity super}. Take $u_M=\min\{u,M\}$ as above.
Then $u_M$ is $(s,p)$-superharmonic by Lemma \ref{lem:minsup}. 
To show that $u_M$ is an $(s,p)$-viscosity supersolution, take $\phi\in C^2(B_r(x_0))$ such that $\phi(x_0)=u_M(x_0)$, $\phi\leq u_M$ in $B_r(x_0)$ and that either (a) or (b) with $\beta > \frac{sp}{p-1}$ in Definition \ref{def:viscosity super}(iii) holds. We need to prove that 
\begin{equation}\label{eq:Lpos}
\L \phi_r(x_0)\geq 0, \quad
\phi_{r}=\begin{cases}
\phi & \text{ in }B_r(x_0),\\
u_M &\text{ in }\ern\setminus B_r(x_0).
\end{cases}
\end{equation}
We argue towards a contradiction, assuming that~\eqref{eq:Lpos} fails.
Then, in particular, $\phi_r \in L^{p-1}_{sp}(\ern)$ since $\phi_r \leq M$, and thus $(\phi_r)_- \notin L^{p-1}_{sp}(\ern)$ would imply $\L \phi_r(x_0) = +\infty$.
By the continuity of $\L \phi_r$ (cf. Lemma \ref{lemma:Lphi cont}), we have $\L \phi_r \leq -\tau$ in $B_{\rho'}(x_0)$ for some $\tau>0$ and $\rho'\in (0,r)$. Moreover, by Lemma \ref{lemma:Lphi cont in eps} there exist $\theta>0$, $\rho \in (0,\rho')$, and $\eta \in C^{2}_0(B_{\rho/2}(x_0))$
with $0 \leq \eta \leq 1$, $\eta(x_0)=1$ such that $\psi:=\phi_r+\theta\eta$ satisfies
\[
\sup_{B_\rho(x_0)}|\L \phi_r-\L \psi|<\tau,
\]
and consequently
\[
\L \psi \leq 0 \quad \text{in } B_\rho(x_0).
\]
By Lemma \ref{lemma:super weak super} we have that $\psi$ is a weak subsolution in $B_\rho(x_0)$, 
and further $(s,p)$-subharmonic in $B_\rho(x_0)$ by continuity, according to Theorem \ref{thm_supersuper2}.
In addition, we have $\psi = \phi_r \leq u_M$ in $\R^n\setminus B_{\rho/2}(x_0)$. Now, $\psi \leq u_M$ in $B_{\rho/2}(x_0)$ by the comparison between $(s,p)$-superharmonic and $(s,p)$-subharmonic functions, i.e. Theorem \ref{comp principle}, together with the semicontinuity of $\psi$ and $u_M$ up to the boundary of $B_{\rho/2}(x_0)$.
This contradicts $\psi(x_0)=\phi(x_0)+\theta>u_M(x_0)$, verifying \eqref{eq:Lpos}. Thus, $u_M$ is an $(s,p)$-viscosity supersolution for any $M>0$ implying that $u$ itself is an $(s,p)$-viscosity supersolution, by Lemma \ref{lem:um}.

\smallskip

Assume now that $u$ is an $(s,p)$-viscosity supersolution in $\Omega$. The only property we need to verify is the comparison property (iii) of Definition \ref{def:superharmonic}. Let $D \Subset \Omega$ be an open set and take $v\in C(\overline D)$ to be a weak solution in $D$ such that $u\geq v$ on $\dd D$ and almost everywhere in $\R^n\setminus D$. It follows by the same arguments as in the first part of the proof that $v$ is an $(s,p)$-viscosity subsolution in $D$. From Theorem \ref{thm:comparison} we can conclude that $u\geq v$ in $D$, and thus $u$ is $(s,p)$-superharmonic.
\end{proof}

\begin{proof}[~Proof of Theorem \ref{thm:equivalence}] The proof now follows from Theorem \ref{thm:equivalence0}, Theorem \ref{thm_supersuper}, and Theorem \ref{thm_supersuper2}.
\end{proof}

\begin{proof}[~Proof of Theorem \ref{thm:viscosity super}] The proof follows from Theorem \ref{thm:equivalence0} and \cite[Theorem 1]{DKKP15}.
\end{proof}

\section{Acknowledgements}
Janne Korvenp\"a\"a has been supported by the Magnus Ehrnrooth Foundation (grant no. ma2014n1, ma2015n3). Tuomo Kuusi is supported by the Academy of Finland.  Erik Lindgren is supported by the Swedish Research Council, grant no. 2012-3124.

\bibliographystyle{amsrefs}
\bibliography{ref} 

\end{document}